\documentclass[a4paper,12pt]{article} 
\usepackage[utf8]{inputenc}
\usepackage[T1]{fontenc}
\usepackage{amsmath,amssymb,amsthm}
\usepackage[]{authblk}
\usepackage{bbm}
\usepackage{color}
\usepackage{bbm}
\usepackage{verbatim}
\usepackage[colorlinks]{hyperref}
\usepackage[capitalise]{cleveref}

\usepackage{pgf,tikz}
\usetikzlibrary{arrows}

\newtheorem{thm}{Theorem}
\crefname{thm}{Theorem}{Theorems}

\crefname{cor}{Corollary}{Corollaries}
\newtheorem{lem}[thm]{Lemma}
\crefname{lem}{Lemma}{Lemmas}
\newtheorem{prop}[thm]{Proposition}
\crefname{prop}{Proposition}{Propositions}

\crefname{conj}{Conjecture}{Conjectures}

\crefname{ques}{Question}{Questions}
\theoremstyle{definition}

\crefname{defn}{Definition}{Definitions}
\newtheorem{rem}[thm]{Remark}
\crefname{rem}{Remark}{Remarks}
\newtheorem{ex}[thm]{Example}
\crefname{ex}{Example}{Examples}

\crefname{obs}{Observation}{Observations}

\crefname{claim}{Claim}{Claims}

\crefname{ass}{Assumption}{Assumptions}

\numberwithin{thm}{section}

\newcommand{\cE}{\ensuremath{\mathcal E}}

\newcommand{\bbR}{{\ensuremath{\mathbb R}} }

\newcommand{\diri}{\mathcal{E}}

\renewcommand{\a}{{\ensuremath{\alpha}}}
\let\oldd\d
\renewcommand{\d}{{\ensuremath{\delta}}}

\newcommand{\E}{{\mathcal{E}}}

\let\oldk\k
\renewcommand{\k}{{\ensuremath{\kappa}}}

\let\oldl\l
\renewcommand{\l}{{\ensuremath{\lambda}}}
\let\oldL\L
\renewcommand{\L}{{\ensuremath{\Lambda}}}

\let\oldo\o
\renewcommand{\o}{{\ensuremath{\omega}}}
\let\oldO\O
\renewcommand{\O}{{\ensuremath{\Omega}}}

\let\oldr\r
\renewcommand{\r}{{\ensuremath{\rho}}}
\newcommand{\s}{{\ensuremath{\sigma}}}
\let\oldt\t
\renewcommand{\t}{{\ensuremath{\tau}}}
\let\oldu\u
\renewcommand{\u}{{\ensuremath{\upsilon}}}

\renewcommand{\leq}{\leqslant}
\renewcommand{\geq}{\geqslant}
\renewcommand{\le}{\leqslant}
\renewcommand{\ge}{\geqslant}
\renewcommand{\to}{\rightarrow}

\DeclareMathOperator*{\argmin}{arg\,min}

\begin{document}
\title{The normal contraction property for non-bilinear Dirichlet forms}
\author{Giovanni Brigati\thanks{\textsf{brigati@ceremade.dauphine.fr}} }
\author{Ivailo Hartarsky\thanks{\textsf{hartarsky@ceremade.dauphine.fr}}}
\affil{CEREMADE, CNRS, Universit\'e Paris-Dauphine, PSL University\protect\\75016 Paris, France}
\date{\vspace{-0.25cm}\today}
\maketitle
\vspace{-0.75cm}
\begin{abstract}
We analyse the class of convex functionals $\cE$ over $\mathrm{L}^2(X,m)$ for a measure space $(X,m)$ introduced by Cipriani and Grillo \cite{cipriani2003nonlinear} and generalising the classic bilinear Dirichlet forms. We investigate whether such non-bilinear forms verify the normal contraction property, i.e., if $\cE(\phi \circ f) \leq \cE(f)$ for all $f \in \mathrm{L}^2(X,m)$, and all 1-Lipschitz functions $\phi: \mathbb R \to \mathbb R$ with $\phi(0)=0$. We prove that normal contraction holds if and only if $\cE$ is symmetric in the sense $\cE(-f) = \cE(f),$ for all $f \in \mathrm{L}^2(X,m).$  An auxiliary result, which may be of independent interest, states that it suffices to establish the normal contraction property only for a simple two-parameter family of functions $\phi$.
\end{abstract}


\section{Introduction}
\label{sec:intro}
\subsection{Setting}
\label{subsec:intro.disucssion}
\subsubsection{Bilinear Dirichlet forms}
Bilinear Dirichlet forms are a well-established topic, related to the theory of Markov processes and semigroups, see \cite{bouleau1991dirichlet,fukushima2011dirichlet,ma2012introduction}. Let $X$ be a nonempty set, let $\mathcal{F}$ be a $\sigma-$algebra over $X$, and take a $\sigma-$finite measure $m: \mathcal{F} \to [0,\infty]$. Let $\Lambda : D(\Lambda) \times D(\Lambda) \to \mathbb R,$ be a symmetric, bilinear, and positive semi-definite form, such that $D(\Lambda) \subset \mathrm{L}^2(X,m)$ is dense.
If the form is \textit{closed}, there exists a unique self-adjoint, positive operator $A : D(A) \to \mathrm{L}^2(X,m),$ such that $D(A) \subset D(\Lambda),$ and 
\[\langle A \, f,g\rangle
= \Lambda(f,g), \quad \forall f \in D(A
),g\in D(\Lambda).\]
Adopting the notation of functional calculus, we also have the formulae $D(\Lambda) = D(A^{1/2})$, and $\Lambda(f,g) = \langle A^{1/2} \, f, \, A^{1/2} \, g \rangle, \quad \forall f,g \in D(\Lambda).$ The bilinear form $\Lambda$ is called a (bilinear) Dirichlet form if 
\[\Lambda(1 \wedge f \vee 0, 1 \wedge f \vee 0) \leq \Lambda(f,f), \quad \forall f\in D(\Lambda).\]
By extension, the term Dirichlet form also refers to the quadratic form 
\[\E(f) = \begin{cases}
     \frac{1}{2} \Lambda(f,f), \quad {if } f \in D(\Lambda); \\
     +\infty, \quad \text{otherwise;}
\end{cases}\]
associated with a bilinear Dirichlet form $\Lambda.$ 
This functional turns out to be always non-negative, convex (since it is quadratic), and lower semicontinuous. Moreover, the subdifferential satisfies $\partial \E = A$.
\subsubsection{Non-bilinear Dirichlet forms}
We next turn to defining non-bilinear Dirichlet forms as they will be studied in the present work. Let $\cE : \mathrm{L}^2(X,m) \to [0,\infty]$ be a convex and l.s.c.\ functional. In all the paper we assume that $\E$ is not the constant $+\infty$. Let $(T_t)_{t\geq0}$ be the semigroup of nonlinear operators generated by $-\partial \cE$, where $\partial$ denotes the subdifferential operator, via the differential equation
\begin{equation}
\label{gradientflow}
\begin{cases}
    \partial_t T_t \,f  \in - \partial \E( T_t \, f),&\forall t \in (0,\infty), \quad \forall u \in \mathrm{L}^2(X,m),\\
    T_0 \,f = f,&\forall f \in \mathrm{L}^2(X,m). 
\end{cases}    
\end{equation}
\Cref{gradientflow} is well-posed for all $f\in\mathrm{L}^2(X,m)$. Its solution is usually called the gradient flow of $\E$ starting at $f$. See \cite{ambrosio2008gradient,brezis1973operateurs} and refer to \cref{subsec:biback} for more background.

We say that a non-negative l.s.c.\ functional $\E$ is a \emph{non-bilinear Dirichlet form} if $\E$ is convex and, for all $t \geq 0,$ the operator $T_t: \mathrm{L}^2(X,m) \to \mathrm{L}^2(X,m)$ verifies
\begin{enumerate}
    \item \emph{order preservation}: $T_t \, f \leq T_t \, g$ for all $f,g \in \mathrm{L}^2(X,m)$ such that $f \leq g$ (for the pointwise order  up to a negligible set); 
    \item \emph{$\mathrm{L}^\infty$-contraction}: $\|T_t \, f - T_t \, g \|_\infty \leq \|f-g\|_\infty$ for all $f,g \in \mathrm{L}^2(X,m).$
\end{enumerate}
This class of forms was introduced by Cipriani and Grillo \cite{cipriani2003nonlinear} and we will provide an equivalent ``static'' definition in \cref{thm1} without reference to the underlying semigroup (also see \cref{th:cipriani}).

Our main goal is to verify the normal contraction property for non-bilinear Dirichlet forms. A \emph{normal contraction} is a $1-$Lipschitz function $\phi: \mathbb R \to \mathbb R,$ such that $\phi(0)=0.$ We denote by $\Phi$ the set of all normal contractions. We say that a functional $\E$ over $\mathrm{L}^2(X,m)$ has the \emph{normal contraction property} if
\begin{equation}
    \label{eq:normalcontraction}
\E(\phi(f)) \leq \E(f),\quad \forall \phi \in \Phi,\quad  \forall f \in \mathrm{L}^2(X,m).    
\end{equation}
In the literature this property goes also under the name of \textit{Second Beurling-Deny Criterion} since \cite{reed1978methods}.
\subsection{Background}
\subsubsection{Bilinear setting}\label{subsec:biback}
Aside their interest in probability, for which we refer to the bibliography of \cite{bouleau1991dirichlet,fukushima2011dirichlet,ma2012introduction}, bilinear Dirichlet forms are also well-linked with linear diffusion equations and semigroups, see \cite{bakry1985diffusions,dolbeault2008bakry}. This link gave fruitful results in the theory of metric measure spaces, allowing for an intrinsic/Eulerian approach towards Ricci curvature bounds, \cite{ambrosio2015bakry}. Under mild hypotheses, the authors of \cite{ambrosio2015bakry} could represent any bilinear Dirichlet form $\E$ as a quadratic Cheeger's energy on the base space $X$. One important point is that Ambrosio, Gigli, and Savar\'e were able to create an appropriate notion of distance $d_{\E}$ directly from the Dirichlet form $\E$. Then, via a condition \textit{\`a la Bakry-Emery,} on the \textit{carr\'e du champ} associated with the quadratic form $\E$, the authors give a sense to notions such as Bochner's inequality or a lower bound on the Ricci curvature. Their approach is equivalent to that of Lott and Villani \cite{lott2009ricci} and Sturm \cite{sturm2006geometryI,sturm2006geometryII}, based on optimal transport. 
The creation of a distance from a bilinear form is a technique present also in \cite{biroli1995sobolev}. 
Bilinear Dirichlet forms also play a role in potential and capacity theory, see \cite{fukushima2011dirichlet,schmidt2017energy}. 

Historically, bilinear Dirichlet forms have been introduced by Beurling and Deny in \cite{beurling1959dirichlet}.
One motivation behind their definition was the fact that being a bilinear Dirichlet form was sufficient to have the normal contraction property (see \cref{eq:normalcontraction}). The fact that controlling one normal contraction is necessary and sufficient to control all of them is nowadays known as the \textit{Beurling-Deny criterion}.
To prove such a property, one usually approximates the function $f$ with weighted sums of characteristic functions. The normal contraction property is a cornerstone for many purposes. For instance, for the development of a differential calculus \cite{ambrosio2015bakry} and the classification of linear Markov semigroups \cite{fukushima2011dirichlet}, both based on bilinear Dirichlet forms.

\subsubsection{Non-bilinear setting}
Generalising the concept of Dirichlet form to a non-bilinear setting is a more recent problem, started with the two works \cite{biroli2005strongly,cipriani2003nonlinear}. A different kind of generalisation is that of \cite{jost1998nonlinear}, but we will not focus on it, since its purpose is different. Using instruments from \cite{barthelemy1996invariance,benilan1979quelques,brezis1973operateurs}, Cipriani and Grillo \cite{cipriani2003nonlinear} provided two equivalent definitions of a non-bilinear Dirichlet form relevant to us, which will be discussed in further detail in \cref{sec:characterisation}. In \cite{cipriani2003nonlinear}, a number of properties of the class of non-bilinear Dirichlet forms are given, in particular with respect to $\Gamma-$convergence (see \cite{dal1993introduction}). 

Two recent works on the topic are \cite{claus2021nonlinear,claus2021energy}, where Claus recovers many structural properties for non-bilinear Dirichlet forms, among which we find a nonlinear Beurling--Deny principle, see \cite[Theorem 2.39]{claus2021nonlinear}.  In the following sections, he develops a nonlinear theory of capacity. Furthermore, in \cite[Corollary 2.40]{claus2021nonlinear} (also see \cite[Theorem 3.22]{claus2021energy}), the normal contraction property is proved for non-bilinear Dirichlet forms, but only for \emph{non-decreasing} normal contractions and additionally assuming that the form is $0$ at $0$ (we avert the reader that in \cite[Definition 2.31]{claus2021nonlinear} non-decreasing normal contractions are named simply normal contractions).

\paragraph{Examples} Let us mention two classes of basic examples, which generalise corresponding families of local and nonlocal bilinear Dirichlet forms. These lie at the core of the functionals analysed in the references quoted at the end of the section. Let $\Omega$ be an open subset of $\mathbb{R}^d$ and $f:\O\times\bbR^d\to\bbR$ be a Borel-measurable function.
Let
\begin{equation}
\label{eq:E:from:f}
\E(u) = \begin{cases}
        \int_\Omega f(x,Du) \, \mathrm{d}x&u \in \mathrm{W}^{1,2}_{\mathrm{loc}}(\Omega), \\
        +\infty&\text{otherwise.}
        \end{cases} 
\end{equation}
We have that $\E$ is a non-bilinear Dirichlet form if $f$ is non-negative, measurable in the first argument and convex and lower-semicontinuous in the second one. See \cite{de1983lower} for the lower semicontinuity of the functional, while the property of being a non-bilinear Dirichlet form can be inferred as in \cite[Theorem 4.1]{cipriani2003nonlinear}. In addition, $\E$ is symmetric if $f(\cdot,-v) = f(\cdot,v),$ for all $v \in \mathbb{R}^d.$ Finally, $\E$ is always local, due to the locality of $Du$ and the fact that $\E$ is an integral functional.
Among local forms, we can consider the following. 
\begin{ex}\label{ex:nec}
Let $\Omega = \mathbb{R}.$ Let $f(x,v) = \max(v,0)
$.
Then, the integral functional $\E$  associated to $f$ by \cref{eq:E:from:f} is a non-symmetric non-bilinear Dirichlet form, which does not satisfy the normal contraction property \cref{eq:normalcontraction} for the function $\phi=-\mathrm{id}$.
\end{ex}
In this class of local functionals we also have the distinguished subclass of Finsler metrics, where 
\[f(x,\cdot) = \|\cdot\|_x, \quad \forall x \in \Omega.\]
The form is bilinear if and only if, for all $x\in\Omega$, the norm $\|\cdot\|_x$ satisfies the parallelogram identity, see \cite[Chapter 5]{brezis2011functional}. 

Some non-local non-bilinear Dirichlet forms appear in \cite{creo2021fractional}, for example. In general we can say that any functional $\E$ of the form 
\[\E(u) = \int_{\Omega^2} \psi(u(x)-u(y)) \, \mathrm{d}x\,\mathrm{d}y,\quad \forall u \in \mathrm{L}^2(\Omega,\mathrm{d}x).\]
is a non-bilinear Dirichlet form for non-negative, l.s.c., convex $\psi$ such that $\psi(0) = 0.$ Lower semicontinuity of the functional comes from Fatou's Lemma, its convexity from the convexity of $\psi$. Finally, one can repeat the computations in \cite[Theorem 2]{hurtado2020non} to prove order-preservation and $\mathrm{L}^\infty-$contraction for the semigroup associated with $\E.$ 

In \cite{cipriani2003nonlinear}, some interesting examples are developed in detail, ranging from functionals from the calculus of variations to Sobolev seminorms in the context of $C^\star-$monomodules. The theory of \cite{cipriani2003nonlinear} can be applied to nonlinear diffusion equations (see \cite{creo2021fractional,feo2021anisotropic} and the references therein), analysis on graphs \cite{hofmann2021spectral,mugnolo2014semigroup}, and analysis on spaces with a very irregular geometry \cite{hinz2020sobolev,meinert2020partial}. Furthermore, Cheeger's energies on extended metric spaces are known to be non-bilinear Dirichlet forms \cite{ambrosio2014calculus}. We refer to \cite{ambrosio2013density,ambrosio2014metric,ambrosio2015bakry} for this theory, which originates from \cite{cheeger1999differentiability,shanmugalingam2000newtonian}. See also \cite{kell2016q,luise2021contraction} for more estimates and contraction properties of Cheeger's energies.

\subsection{Main results}
\label{subsec:intro.mainresults}
Our main result is the following.
\begin{thm}\label{th:contraction}
Let $\E$ be a non-bilinear Dirichlet form. Then $\E$ has the normal contraction property \cref{eq:normalcontraction} if and only if
\begin{equation}
    \label{eq:sym}
    \E(-f) \leq \E(f)\quad\forall f \in \mathrm{L}^2(X,m).
\end{equation}
\end{thm}
This theorem goes in the same direction as the well-established one for the bilinear case \cite{bouleau1991dirichlet,fukushima2011dirichlet,ma2012introduction}. We merely prove that a form will operate on all normal contractions, once it operates on the simplest one. Henceforth, we say that $\E$ is \emph{symmetric} if \cref{eq:sym} holds and, equivalently, $\cE(-f)=\cE(f)$ for all $f\in \mathrm{L}^2(X,m)$. As witnessed by \cref{ex:nec}, the necessary symmetry assumption \cref{eq:sym} needs to be made, since this non-bilinear Dirichlet form does not have the normal contraction property.

Let us highlight that \cref{th:contraction} may be viewed as a strengthening of the result of Claus \cite[Corollary 2.40]{claus2021nonlinear}, whose proof follows the far more conventional approach of \cite{benilan1991completely,benilan1979quelques}. The class of normal contractions we consider is richer and it controls, for example, the absolute value of the argument of the non-bilinear Dirichlet form, which can be very useful (see e.g.\ \cite{fukushima2011dirichlet}), as well as more complicated contractions.

In order to prove \cref{th:contraction}, we establish two results, both of which may be of independent interest. Firstly, we provide an equivalent characterisation of non-bilinear Dirichlet forms, which turns out to be more widely for our purposes than the other equivalent static characterisation of \cite[Theorem 3.8]{cipriani2003nonlinear}, recalled in \cref{th:cipriani}. To do so, we require a bit of notation. For all $f,g \in \mathrm{L}^2(X,m)$, and $\a\in[0,\infty)$ we denote by $f\vee g$ and $f\wedge g$ denote the pointwise maximum and minimum and set $H_\alpha(f,g)=(g-\alpha)\vee f\wedge(g+\alpha)$ (see \cref{fig:H}), that is,
\begin{equation}
\label{eq:def:H}
H_\a(f,g)(x)=\begin{cases}
g(x)-\a&f(x)-g(x)<-\a,\\
f(x)&f(x)-g(x)\in[-\a,\a],\\
g(x)+\a&f(x)-g(x)>\a.
\end{cases}
\end{equation}

\begin{figure}
    \centering
\begin{tikzpicture}[line cap=round,line join=round,>=triangle 45,x=0.7cm,y=.7cm]
\draw[->] (-2,0) -- (4,0) node[below] {$f(x)$};
\draw[->] (0,-1) -- (0,4) node[below right] {$H_\a(f,g)(x)$};
\draw[shift={(2,0)}] (0pt,2pt) -- (0pt,-2pt) node[below] {$g(x)$};
\draw[shift={(1,0)}] (0pt,2pt) -- (0pt,-2pt);
\draw[shift={(3,0)}] (0pt,2pt) -- (0pt,-2pt);
\draw (0pt,2pt) -- (0pt,-2pt) node[below left] {$0$};
\draw[<->] (1,0.5)--(3,0.5) node[above,midway] {$2\a$};
\draw[very thick] (-2,1)-- (1,1)--(3,3)--(4,3);
\draw[dashed] (0,0)--(1,1);
\end{tikzpicture}
    \caption{Graph of the function $H_\a(f,g)(x)$ for fixed $g(x)$.}
    \label{fig:H}
\end{figure}
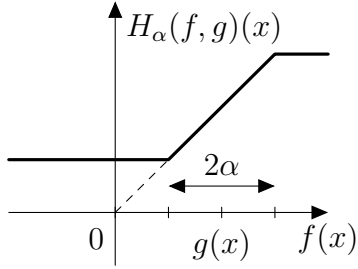

\begin{thm}\label{thm1}
Let $\E:\mathrm{L}^2(X,m)\to[0,\infty]$ be a 
l.s.c.\ functional. 
Then, $\mathcal E$ is a non-bilinear Dirichlet form if and only if, for all $f,g \in \mathrm{L}^2(X,m)$, and $\a\in[0,\infty)$, $\E$ verifies 
\begin{align}
\label{eq:minmax}
\cE(f\vee g)+\cE(f\wedge g)\le{}&\cE(f)+\cE(g),\\
\label{eq:HK}
\cE(H_\a(f,g))+\cE(H_\a(g,f))\le{}&\cE(f)+\cE(g).
\end{align}
\end{thm}
The advantage of \cref{thm1} as compared to \cref{th:cipriani} is that conditions \cref{eq:HK,eq:minmax} are easier to verify and useful to develop other functional inequalities such as the normal contraction property \cref{eq:normalcontraction}.

The second important step towards \cref{th:contraction} is a reduction.
\begin{lem}\label{lemma1}
Let $G$ be the set of all normal contractions $\phi \in \Phi$ such that $|\phi'|=1$ and $\phi'$ has at most two points of discontinuity.  Let $\langle G \rangle$ be the collection of all finite compositions of elements in $G.$ Then, $\langle G \rangle$ is dense in $\Phi$ for the pointwise convergence on $\mathbb R$. 
\end{lem}
We observe that the elements of $G$ are irreducible with respect to composition, so that $G$ is minimal in this sense. While the space $G$ is quite simple, proving the normal contraction property \cref{eq:normalcontraction} for $\phi\in G$ by hand from symmetry and \cref{eq:minmax,eq:HK} is still delicate, albeit elementary.

\subsection{Plan of the paper}
\label{subsec:intro.plan}
The remainder of the paper is structured as follows. In \cref{sec:characterisation}, we establish \cref{thm1}. In \cref{sec:contraction}, we prove \cref{th:contraction}, relying on \cref{thm1}. This is the heart of our work. Finally, we discuss future directions of research in \cref{sec:future}.

\section{Efficient equivalent characterisation of non-bilinear Dirichlet forms}
\label{sec:characterisation}
The goal of the present section is to prove \cref{thm1}.

\subsection{Preliminaries}
We introduce the subsets $C_1$ and $C_{2,\alpha}$, for $\alpha \in [0,\infty),$ of $\mathrm{L}^2(X,m;\mathbb{R}^2):$
\begin{align}
\label{eq:c1}
C_1 &{}= \left\{ (f,g) \in \mathrm{L}^2(X,m;\mathbb{R}^2) : f \leq g\right\}, \\
\label{eq:c2a}
C_{2,\alpha} &{}= \left\{ (f,g) \in \mathrm{L}^2(X,m;\mathbb{R}^2) : |f-g| \leq \alpha\right\}.
\end{align}
We notice that for all $\alpha$, the sets $C_1$ and $C_{2,\alpha}$ are convex and closed in the $\mathrm{L}^2-$topology. For any closed and convex subset $C$, the $1-$Lipschitz projection operator $P_C:\mathrm{L}^2(X,m;\mathbb{R}^2) \to C$ is defined by
\[P_C(f,g) = \argmin_{(w,z) \in C} \|f-w\|^2_2 +\|g-z\|^2_2.\]
The projection map sends any point $(f,g)$ to the closest point $P_C(f,g)$ in $C.$
We denote by $P^1_C$ and $P^2_C$ the two components of the projection operator in $\mathrm{L}^2(X,m).$ 
More properties of projection maps are studied in \cite{brezis2011functional}. 
If one considers the sets $C_1$ and $C_{2,\alpha}$, we have an explicit expression for the projections, thanks to \cite[Lemma 3.3]{cipriani2003nonlinear}:
\begin{align}
P_1(f,g) &{}= \left(f-\frac{1}{2}((f-g) \vee 0),g+\frac{1}{2}((f-g)\vee 0)\right),\\
P_{2,\alpha}(f,g) &{}= \left(g+\frac{1}{2}\varphi_\alpha\circ(f-g),f-\frac{1}{2}\varphi_\alpha\circ(f-g)\right),\label{eq:def:P2a}
\end{align}
where $\varphi_\a:\bbR\to\bbR$ is given by
\begin{equation}
\label{eq:def:phia}
\varphi_\alpha(z) = ((z+\alpha) \vee 0) +((z-\alpha)
\wedge 0) .
\end{equation}
We further recall \cite[Definition 3.1, Remark 3.2, Theorem 3.6]{cipriani2003nonlinear}. 
\begin{thm}
\label{th:cipriani}
Let $\E: \mathrm{L}^2(X,m) \to [0,\infty]$ be a l.s.c.\ functional. 
Then $\E$ is a non-bilinear Dirichlet form if and only if, for all $f,g\in\mathrm{L}^2(X,m)$ and $\alpha\in[0,\infty)$, $\E$ verifies
\begin{align}
    \label{eq:prcr1}
    \E\left(P^1_1(f,g)\right) + \E\left(P^2_1(f,g)\right) \leq \E(f) + \E(g),\\
    \label{eq:prcr2}
    \E\left(P^1_{2,\alpha}(f,g)\right) + \E\left(P^2_{2,\alpha}(f,g)\right) \leq \E(f) + \E(g).
\end{align}
\end{thm}
The key argument is the well-known fact from \cite{barthelemy1996invariance,brezis1973operateurs} stating that 
\[\E\left(P_C^1(f,g)\right) + \E\left(P_C^2(f,g)\right) \leq \E(f) + \E(g)\]
for all $f,g\in\mathrm{L}^2(X,m)$
if and only if the semigroup $T_t$ from \cref{gradientflow} preserves $C:$
\[T_t \, C \subset C, \quad \forall t \geq 0,\]
where $C$ can be any convex and closed set. 
Thus, \cref{eq:prcr1,eq:prcr2} correspond to the order-preservation and the $\mathrm{L}^\infty-$contraction properties for $(T_t)_t,$ respectively.
In \cite[Theorem 3.8]{cipriani2003nonlinear} one more step is made. 
\begin{thm}\label{thm:minmax}
Let $\E:\mathrm{L}^2(X,m) \to [0,\infty]$ be a l.s.c.\ functional.
Then, $\E$ satisfies  \cref{eq:minmax} if and only if $\cE$ is convex and satisfies \cref{eq:prcr1}.
\end{thm}
Indeed, the last statement is a consequence of the more general \cite[Proposition 2.5]{barthelemy1996invariance}, which we will also use.
\begin{thm}
\label{thmb3}
Let $C$ be a closed convex subset of $\mathrm{L}^2(X,m;\mathbb{R}^2)$, let $P_C=(P_C^1,P_C^2)$ be the associated orthogonal projection. Let $\diri: \mathrm{L}^2(X,m) \to [0,\infty]$ be a l.s.c.\ functional.
Let $h,k : \mathrm{L}^2(X,m;\mathbb R^2) \to \mathrm{L}^2(X,m)$ be two continuous mappings such that, for all $u,v \in \mathrm{L}^2(X,m)$ and $t,s \in [0,1]$ it holds that
\begin{align}\label{eq10}
    h(u_t,v_s) &{}= u_{1-s}, &
    k(u_t,v_s) &{}= v_{1-t},
\end{align}
where 
\begin{align*}
u_t &{}= (1-t)u+t h(u,v),&
v_s &{}= (1-s)v+s k(u,v).\end{align*}
Moreover, assume 
\begin{equation}\label{eq11}
P_C(u,v) = (u_{1/2},v_{1/2}).
\end{equation}
Then, we have that for all $u,v \in \mathrm{L}^2(X,m)$
\[\cE\left(P_C^{1}(u,v)\right) + \cE\left(P_C^{2}(u,v)\right) \leq \cE(u)+\cE(v),\]
if and only if $\E$ is convex and for all $u,v \in \mathrm{L}^2(X,m)$
\[\cE(h(u,v))+\cE(k(u,v)) \leq \cE(u) + \cE(v).\]
\end{thm}
\begin{rem}
Note that, given \cref{th:cipriani,thm:minmax}, it is easy to deduce that every non-bilinear Dirichlet form satisfies \cref{eq:minmax,eq:HK}, which is the direction of \cref{thm1} we will use for proving \cref{th:contraction}. Indeed,
\[H_\alpha(f,g)=\frac12P^1_{2,\alpha}(f,g)+\frac12P^2_{2,\alpha}(g,f)\]
for all $\alpha\ge 0$ and $f,g\in\mathrm{L}^2(X,m)$, so that convexity and \cref{eq:prcr2} give
\begin{multline*}
\E(H_\a(f,g))+\E(H_\a(g,f))\\\begin{aligned}\le{}& \frac12\left(\E(P^1_{2,\a}(f,g))+\E(P^2_{2,\a}(g,f))+\E(P^1_{2,\a}(g,f))+\E(P^2_{2,\a}(f,g))\right)\\\le{}& \E(f)+\E(g).\end{aligned}
\end{multline*}
\end{rem}

\subsection{Proof of Theorem~\ref{thm1}}
To conclude the section, we show that the convex sets $C_{2,\alpha}$ verify the hypotheses of \cref{thmb3}.

\begin{proof}[Proof of \cref{thm1}]
Fix $\a>0$. Recalling the explicit expression of $\varphi_\a$ from \cref{eq:def:phia}, for any $u,v\in \mathrm{L}^2(X,m)$ we have
\begin{equation*}
    \varphi_\alpha\circ (u-v) (x) = \begin{cases}
                    u(x)-v(x)-\alpha&u(x)-v(x) \leq -\alpha, \\ 
                    2u(x)-2v(x)&|u(x)-v(x)| \leq \alpha, \\
                    u(x)-v(x)+\alpha &u(x)-v(x) \geq \alpha. \\
                    \end{cases} 
\end{equation*}
Further recalling the expression of $P_{2,\alpha}$ from \cref{eq:def:P2a}, in order to satisfy \cref{eq11}, we now choose $h,k : \mathrm{L}^2(X,m;\mathbb{R}^2) \to \mathrm{L}^2(X,m)$ such that 
\begin{align*}
v + \dfrac{1}{2}\varphi_\alpha\circ(u-v) &{}= \dfrac{u + h(u,v)}{2},&
u - \dfrac{1}{2}\varphi_\alpha\circ (u-v) &{}= \dfrac{v + k(u,v)}{2}.
\end{align*}
Therefore, the expressions for $h,k$ are the following
\begin{align*}
    h(u,v)(x) &{}= \begin{cases}
                    v(x)-\alpha&u(x)-v(x) \leq -\alpha, \\ 
                    u(x)&|u(x)-v(x)| \leq \alpha, \\
                    v(x)+\alpha&u(x)-v(x) \geq \alpha, \\
                    \end{cases} 
          \\
    k(u,v)(x) &{}= \begin{cases}
                    u(x)+\alpha&u(x)-v(x) \leq -\alpha, \\ 
                    v(x)&|u(x)-v(x)| \leq \alpha, \\
                    u(x)-\alpha&u(x)-v(x) \geq \alpha,
                    \end{cases}               
\end{align*}
and we notice that $h(u,v)=H_\a(u,v)$ and $k(u,v)= H_\alpha(v,u)$.

It remains to verify the twist condition \cref{eq10}.
Fix $s, t, u, v$ as in the hypothesis. Since the values of $H_\alpha$ is defined pointwise, we also fix $x \in X$ and drop this parameter for compactness of notation. Suppose that $|u - v| \leq \alpha$, then $H(u, v) = u, H(v, u) = v$, so $u_t=u_{1-s}=u$, $v_s=v$. The case $u - v < -\alpha$ is analogous to that with $u - v > \alpha$, since the role of $u$ and $v$ is
symmetric. Hence, we will discuss only the former. Here we have
\begin{align*}
u_t &{}= (1 - t)u + t(v - \alpha),&
v_s &{}= (1 - s)v + s(u + \alpha).
\end{align*}
We need not discuss more subcases for the expression of $H_\alpha(u_t, v_s)$, since
\begin{align*}
u_t - v_s &{}= (1 - t)u + t(v - \alpha) - (1 - s)v - s(u + \alpha) \\
&{}= (1 - t - s)(u - v) -(t+s)\alpha <-\alpha.
\end{align*}
Hence,
\[H_\alpha(u_t, v_s) = v_s - \alpha = (1 - s)v + su + (s - 1)\alpha=u_{1-s},\]
The second condition in \cref{eq10} follows similarly, so we omit it. Thus, applying \cref{thmb3}, \cref{eq:prcr2} is equivalent to $\cE$ being convex and \cref{eq:HK}. Yet, \cref{thm:minmax} gives that the convexity and \cref{eq:prcr1} are equivalent to \cref{eq:minmax}, so \cref{thm1} reduces to \cref{th:cipriani}.
\end{proof}

\section{The normal contraction property}
\label{sec:contraction}
Throughout this section we fix a measure space $(X,m)$ and a functional on $\mathrm{L}^2(X,m)$ satisfying  symmetry and \cref{eq:minmax,eq:HK} for all $f,g\in\mathrm{L}^2(X,m)$ and $\a\in[0,\infty)$.

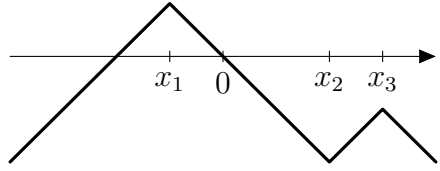
\begin{figure}
    \centering
\begin{tikzpicture}[line cap=round,line join=round,>=triangle 45,x=0.7cm,y=.7cm]
\draw[->] (-4,0) -- (4,0);
\draw[shift={(-1,0)}] (0pt,2pt) -- (0pt,-2pt) node[below] {$x_1$};
\draw (0pt,2pt) -- (0pt,-2pt) node[below] {$0$};
\draw[shift={(2,0)}] (0pt,2pt) -- (0pt,-2pt) node[below] {$x_2$};
\draw[shift={(3,0)}] (0pt,2pt) -- (0pt,-2pt) node[below] {$x_3$};
\draw[very thick] (-4,-2)-- (-1,1)--(2,-2)--(3,-1)--(4,-2);
\end{tikzpicture}
    \caption{Graph of the function $\phi_{x_1,x_2,x_3}$.}
    \label{fig:phi}
\end{figure}
We will prove the normal contraction property \cref{eq:normalcontraction} progressively, starting from simple functions $\phi$. More specifically, for $k\in\{0,1,\dots\}$, $x_1,\dots,x_k\in\bbR$ such that $-\infty=x_0<x_1<\dots<x_k<x_{k+1}=\infty$, we consider the continuous function $\phi_{x_1,\dots,x_k}:\bbR\to\bbR$ (see \cref{fig:phi}) defined by $\phi_{x_1,\dots,x_k}(0)=0$ and
\begin{equation}
\label{eq:def:phi}\phi'_{x_1,\dots,x_k}(x)=(-1)^{i}
\end{equation}
for $x\in(x_i,x_{i+1})$. Let us denote $F_k=\{\phi_{x_1,\dots,x_k}:x_1<\dots<x_k \in\bbR\}$, so that $F_0=\{\mathrm{id}\}$. We further set $\Phi_{x_1,\dots,x_k}=\cE\circ\phi_{x_1,\dots,x_k}$.

\subsection{Basic contractions}
\begin{prop}
\label{prop:one-cusp}
For any $x\in\bbR$ and $f\in\mathrm{L}^2(X,m)$ we have $\Phi_{x}(f)\le \cE(f)$.
\end{prop}
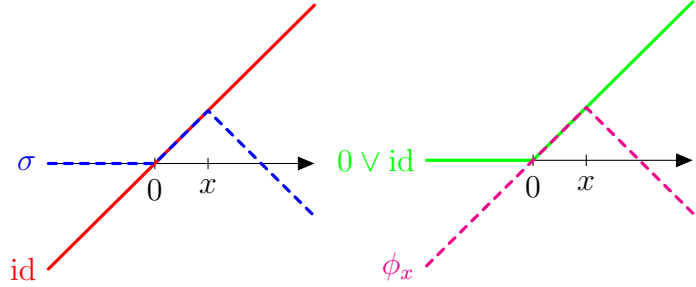
\begin{figure}
    \centering
\begin{tikzpicture}[line cap=round,line join=round,>=triangle 45,x=0.7cm,y=.7cm]
\draw[->] (-2,0) -- (3,0);
\draw[shift={(1,0)}] (0pt,2pt) -- (0pt,-2pt) node[below] {$x$};
\draw (0pt,2pt) -- (0pt,-2pt) node[below] {$0$};
\draw[color=red,very thick] (-2,-2) node[left] {$\mathrm{id}$}--(3,3);
\draw[color=blue,very thick,dashed] (-2,0) node[left] {$\s$}--(0,0)--(1,1)--(3,-1);
\end{tikzpicture}
\begin{tikzpicture}[line cap=round,line join=round,>=triangle 45,x=0.7cm,y=0.7cm]
\draw[->] (-2,0) -- (3,0);
\draw[shift={(1,0)}] (0pt,2pt) -- (0pt,-2pt) node[below] {$x$};
\draw (0pt,2pt) -- (0pt,-2pt) node[below] {$0$};
\draw[color=green,very thick] (-2,0) node[left] {$0\vee\mathrm{id}$}--(0,0)--(3,3);
\draw[color=magenta,very thick,dashed] (-2,-2) node[left] {$\phi_x$}--(1,1)--(3,-1);
\end{tikzpicture}
\caption{Illustration of \cref{eq:minmax4}.\label{fig:minmax4}}
\end{figure}
\begin{proof}
Fix $x\ge 0$ (the case $x<0$ is treated identically) and $f$. By \cref{eq:minmax}
\begin{equation}
\label{eq:minmax4}\Phi_x(f)+\cE(0\vee f)\le \cE(f)+\cE(\s\circ f)\end{equation}
(see \cref{fig:minmax4}), where 
\[\s(y)=\begin{cases}0&y\le 0,\\
y&y\in(0,x),\\
2x-y&y\ge x.
\end{cases}\]
Thus, it suffices to show that $\cE(0\vee f)\ge \cE(\s\circ f)$.

But symmetry and \cref{eq:HK} with $\a=2x$ (see \cref{fig:HK5}) give
\begin{equation}
\label{eq:HK5}
\begin{aligned}2\cE(\s\circ f)\le{}&\cE(\s\circ f)+\cE(-\s\circ f)\\
\le{}& \cE(0\vee f)+\cE(-(0\vee f))\le2\cE(0\vee f),
\end{aligned}
\end{equation}
concluding the proof.
\end{proof}

\begin{figure}
    \centering
\begin{tikzpicture}[line cap=round,line join=round,>=triangle 45,x=0.7cm,y=.7cm]
\draw[->] (-2,0) -- (4,0);
\draw (0pt,2pt) -- (0pt,-2pt) node[below] {$0$};
\draw[color=red,very thick] (-2,0)--(0,0)--(3,3)node[right] {$0\vee\mathrm{id}$};
\draw[color=blue,very thick,dashed] (-2,0)--(0,0)--(3,-3) node[right] {$-(0\vee\mathrm{id})$}; 
\draw[<->] (1,-1)--(1,1) node[midway,above right] {$2x$};
\end{tikzpicture}
\begin{tikzpicture}[line cap=round,line join=round,>=triangle 45,x=0.7cm,y=0.7cm]
\draw[->] (-2,0) -- (4,0);
\draw[shift={(1,0)}] (0pt,2pt) -- (0pt,-2pt) node[below] {$x$};
\draw (0pt,2pt) -- (0pt,-2pt) node[below] {$0$};
\draw[color=green,very thick] (-2,0)--(0,0)--(1,-1)--(3,1) node[right] {$-\s$};
\draw[color=magenta,very thick,dashed] (-2,0)--(0,0)--(1,1)--(3,-1) node[right] {$\s$};
\phantom{\draw[color=blue,very thick,dashed] (-2,0)--(0,0)--(3,-3) node[right] {$-(0\vee\mathrm{id})$};}
\end{tikzpicture}
\caption{Illustration of \cref{eq:HK5}.\label{fig:HK5}}
\end{figure}
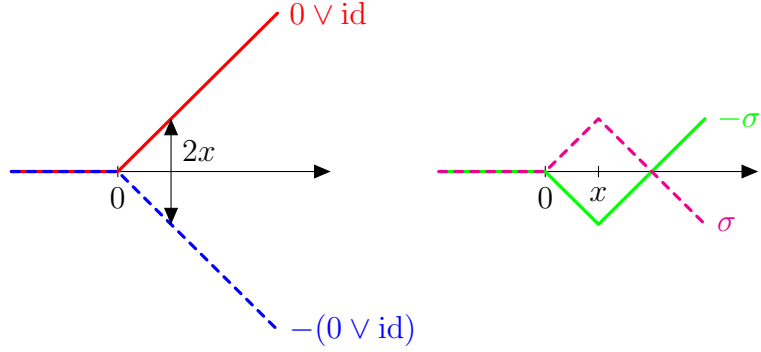

\begin{prop}
\label{prop:one-sided}
For any $0\le x_1<x_2$ or $x_1<x_2\le 0$ and $f\in\mathrm{L}^2(X,m)$ it holds that $\Phi_{x_1,x_2}(f)\le \cE(f)$.
\end{prop}
\begin{figure}
\centering
\begin{tikzpicture}[line cap=round,line join=round,>=triangle 45,x=0.7cm,y=.7cm]
\draw[->] (-4,0) -- (3,0);
\draw[shift={(-1,0)}] (0pt,2pt) -- (0pt,-2pt) node[below] {$x_1$};
\draw[shift={(-2,0)}] (0pt,2pt) -- (0pt,-2pt) node[below] {$0$};
\draw[shift={(2,0)}] (0pt,2pt) -- (0pt,-2pt) node[above] {$x_2$};
\draw[color=red,very thick] (-4,0)-- (-2,0)--(1,3) node[right] {$0\vee\mathrm{id}$};
\draw[color=blue,very thick,dashed]  (-4,0)--(-1,0)--(2,-3)--(3,-2) node[right] {$\s$};
\draw[<->] (-1,0)--(-1,1) node[midway,right] {$x_1$};
\end{tikzpicture}
\begin{tikzpicture}[line cap=round,line join=round,>=triangle 45,x=0.7cm,y=0.7cm]
\draw[->] (-4,0) -- (3,0);
\draw[shift={(-1,0)}] (0pt,2pt) -- (0pt,-2pt) node[below] {$x_1$};
\draw[shift={(-2,0)}] (0pt,2pt) -- (0pt,-2pt) node[below] {$0$};
\draw[shift={(2,0)}] (0pt,2pt) -- (0pt,-2pt) node[above] {$x_2$};
\draw[color=green,very thick] (-4,0)--(-1,0)--(2,3) node[right] {$0\vee(\mathrm{id}-x_1)$};
\draw[color=magenta,very thick,dashed] (-4,0)--(-2,0)--(-1,1)--(2,-2)--(3,-1) node[right] {$\psi$};
\phantom{\draw (0,0)--(0,-3);}
\end{tikzpicture}
\caption{Illustration of \cref{eq:HK4}.\label{fig:HK4}}
\end{figure}
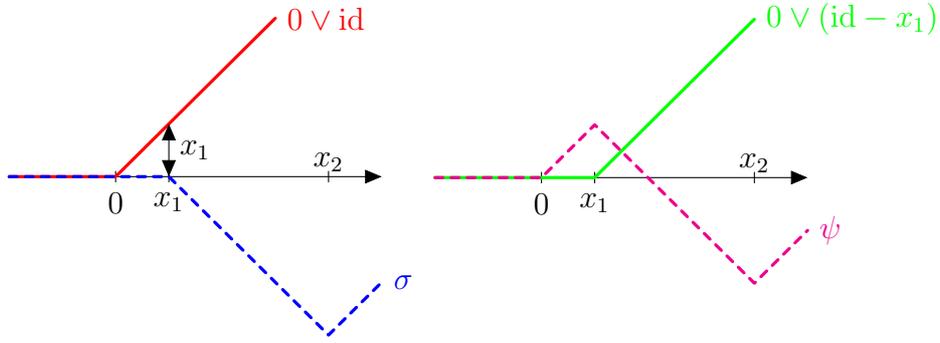
\begin{proof}
Fix $0\le x_1 <x_2$ and $f$, the case $x_1<x_2\le 0$ being analogous. Let 
\begin{align*}
\s(x)&{}=\begin{cases}
0&x\le x_1,\\
x_1-x&x\in(x_1,x_2),\\
x+x_1-2x_2&x\ge x_2,
\end{cases}&
\psi(x)&{}=\begin{cases}
0&x\le 0,\\
\phi_{x_1,x_2}(x)&x>0.\end{cases}
\end{align*}
Then \cref{eq:HK} with $\a=x_1$ (see \cref{fig:HK4}) gives
\begin{equation}
\label{eq:HK4}
\cE(\psi\circ f)+\cE(0\vee (f-x_1))\le \cE(0\vee f)+\cE(\s\circ f).
\end{equation}
Moreover, by symmetry and \cref{eq:HK} for $\a=2(x_2-x_1)$ (see \cref{fig:HK3}) we get
\begin{equation}
\label{eq:HK3}
\begin{aligned}
2\cE(0\vee (f-x_1))\ge{}&\cE(0\vee(f-x_1))+\cE(0\wedge (x_1-f))\\
\ge{}& \cE(\s\circ f)+\cE(-\s\circ f)\ge2\cE(\s\circ f),
\end{aligned}
\end{equation}
so that $\cE(\psi\circ f)\le \cE(0\vee f)$. Furthermore, \cref{eq:minmax} gives
\begin{equation}
\label{eq:minmax3}
\Phi_{x_1,x_2}(f)+\cE(0\vee f)\le \cE(\psi\circ f)+\cE(f)
\end{equation}
(see \cref{fig:minmax3}), yielding the desired conclusion.
\end{proof}
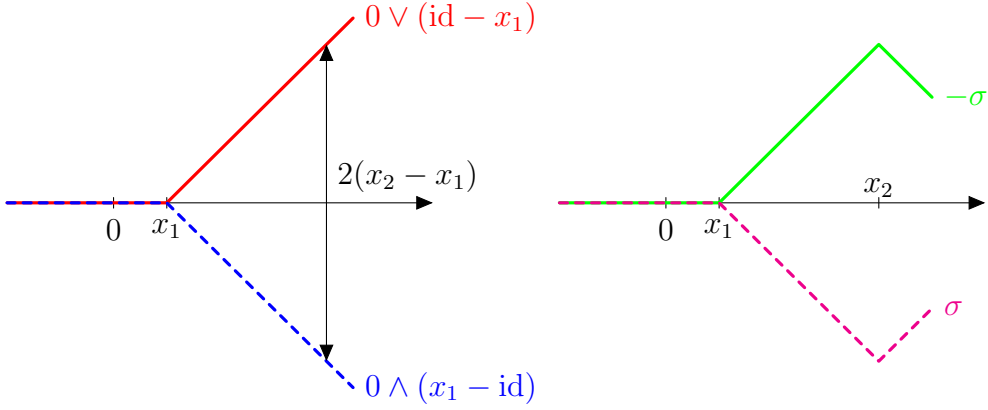
\begin{figure}
    \centering
\begin{tikzpicture}[line cap=round,line join=round,>=triangle 45,x=0.7cm,y=.7cm]
\draw[->] (-4,0) -- (4,0);
\draw[shift={(-1,0)}] (0pt,2pt) -- (0pt,-2pt) node[below] {$x_1$};
\draw[shift={(-2,0)}] (0pt,2pt) -- (0pt,-2pt) node[below] {$0$};
\draw[color=red,very thick] (-4,0)--(-1,0)--(2.5,3.5) node[right] {$0\vee(\mathrm{id}-x_1)$};
\draw[color=blue,very thick,dashed] (-4,0) --(-1,0)--(2.5,-3.5)node[right] {$0\wedge(x_1-\mathrm{id})$};
\draw[<->] (2,3)--(2,-3) node[midway,above right] {$2(x_2-x_1)$};
\end{tikzpicture}
\begin{tikzpicture}[line cap=round,line join=round,>=triangle 45,x=0.7cm,y=0.7cm]
\draw[->] (-4,0) -- (4,0);
\draw[shift={(-1,0)}] (0pt,2pt) -- (0pt,-2pt) node[below] {$x_1$};
\draw[shift={(-2,0)}] (0pt,2pt) -- (0pt,-2pt) node[below] {$0$};
\draw[shift={(2,0)}] (0pt,2pt) -- (0pt,-2pt) node[above] {$x_2$};
\draw[color=green,very thick] (-4,0)--(-1,0)--(2,3)--(3,2) node[right] {$-\s$};
\draw[color=magenta,very thick,dashed] (-4,0)--(-1,0)--(2,-3)--(3,-2) node[right] {$\s$};
\phantom{\draw (0,0)--(0,-3.5) node[right]{$(x_1$};}
\end{tikzpicture}
\caption{Illustration of \cref{eq:HK3}.\label{fig:HK3}}
\end{figure}

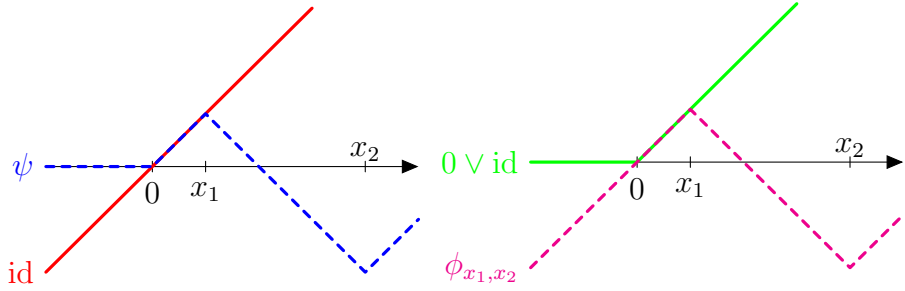
\begin{figure}[ht]
\centering
\begin{tikzpicture}[line cap=round,line join=round,>=triangle 45,x=0.7cm,y=.7cm]
\draw[->] (-4,0) -- (3,0);
\draw[shift={(-1,0)}] (0pt,2pt) -- (0pt,-2pt) node[below] {$x_1$};
\draw[shift={(-2,0)}] (0pt,2pt) -- (0pt,-2pt) node[below] {$0$};
\draw[shift={(2,0)}] (0pt,2pt) -- (0pt,-2pt) node[above] {$x_2$};
\draw[color=red,very thick] (-4,-2) node[left] {$\mathrm{id}$}--(1,3);
\draw[color=blue,very thick,dashed] (-4,0) node[left] {$\psi$}--(-2,0)--(-1,1)--(2,-2)--(3,-1);
\end{tikzpicture}
\begin{tikzpicture}[line cap=round,line join=round,>=triangle 45,x=0.7cm,y=0.7cm]
\draw[->] (-4,0) -- (3,0);
\draw[shift={(-1,0)}] (0pt,2pt) -- (0pt,-2pt) node[below] {$x_1$};
\draw[shift={(-2,0)}] (0pt,2pt) -- (0pt,-2pt) node[below] {$0$};
\draw[shift={(2,0)}] (0pt,2pt) -- (0pt,-2pt) node[above] {$x_2$};
\draw[color=green,very thick] (-4,0) node[left] {$0\vee\mathrm{id}$}-- (-2,0)--(1,3);
\draw[color=magenta,very thick,dashed] (-4,-2) node[left] {$\phi_{x_1,x_2}$}--(-1,1)--(2,-2)--(3,-1);
\end{tikzpicture}
\caption{Illustration of \cref{eq:minmax3}.\label{fig:minmax3}}
\end{figure}

\begin{prop}
\label{prop:two-sided}
For any $x_1<0<x_2$ and $f\in\mathrm{L}^2(X,m)$ it holds that $\Phi_{x_1,x_2}(f)\le\cE(f)$.
\end{prop}
\begin{figure}
\centering
\begin{tikzpicture}[line cap=round,line join=round,>=triangle 45,x=0.7cm,y=.7cm]
\draw[->] (-4,0) -- (3,0);
\draw[shift={(-1,0)}] (0pt,2pt) -- (0pt,-2pt) node[below] {$x_1$};
\draw (0pt,2pt) -- (0pt,-2pt) node[below] {$0$};
\draw[color=red,very thick] (-3,-3) node[left] {$\mathrm{id}$}-- (3,3);
\draw[color=blue,very thick,dashed] (-4,-2) node[left] {$\psi$}--(-1,1)--(2,-2)--(3,-2);
\draw[<->] (2,2)--(2,-2) node[midway,above left] {$2x_2$};
\end{tikzpicture}
\begin{tikzpicture}[line cap=round,line join=round,>=triangle 45,x=0.7cm,y=0.7cm]
\draw[->] (-4,0) -- (3,0);
\draw[shift={(-1,0)}] (0pt,2pt) -- (0pt,-2pt) node[below] {$x_1$};
\draw[shift={(2,0)}] (0pt,2pt) -- (0pt,-2pt) node[below] {$x_2$};
\draw (0pt,2pt) -- (0pt,-2pt) node[below] {$0$};
\draw[color=green,very thick] (-3,-3) node[left] {$\mathrm{id}\wedge x_2$}-- (2,2)--(3,2);
\draw[color=magenta,very thick,dashed] (-4,-2) node[left] {$\phi_{x_1,x_2}$}--(-1,1)--(2,-2)--(3,-1);
\end{tikzpicture}
\caption{Illustration of \cref{eq:HK2}.\label{fig:HK2}}
\end{figure}
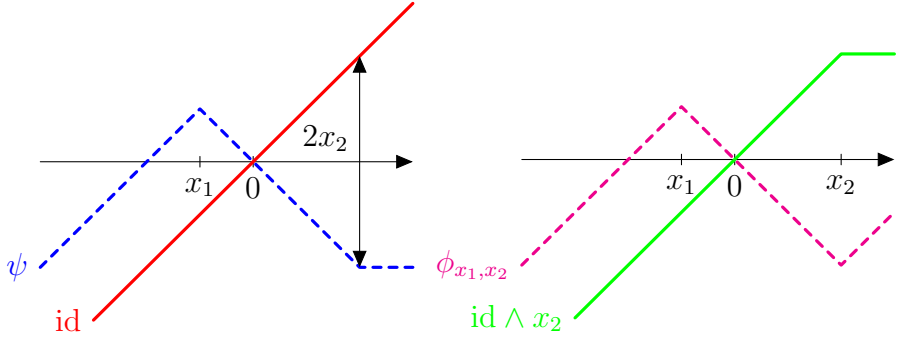
\begin{proof}
Without loss of generality assume that $x_2>-x_1$ and fix $f$. Consider 
\[\psi(x)=\begin{cases}
x-2x_1&x<x_1,\\
-x&x_1\le x\le x_2,\\
-x_2&x>x_2.
\end{cases}\]
Then \cref{eq:HK} with $\a=2x_2$ (see \cref{fig:HK2}) gives
\begin{equation}
\label{eq:HK2}
\Phi_{x_1,x_2}(f)+\cE(f\wedge x_2)\le \cE(f)+\cE(\psi\circ f).
\end{equation}
Yet, $\psi=\phi_{x_1}\circ(\mathrm{id}\wedge x_2)$, so by \cref{prop:one-cusp} we have $\cE(\psi\circ f)\le \cE(f\wedge x_2)$. Combining this with \cref{eq:HK2} yields the desired conclusion.
\end{proof}

\subsection{Reduction to basic contractions}
As we will see, the next proposition is essentially \cref{lemma1}.
\begin{prop}
\label{prop:reduction}
Any $\phi\in F_k$ with $k\ge 0$ can be written as $\phi^1\circ\dots\circ\phi^{\lfloor k/2\rfloor}\circ \psi$ with $\phi^i\in F_2$ for all $i\in\{1,\dots,\lfloor k/2\rfloor\}$ and $\psi\in F_{k-2\lfloor k/2\rfloor}$.
\end{prop}
\begin{proof}
We proceed by induction on $k$. The statement is trivial for $k\in\{0,1,2\}$. Assume that $\phi=\phi_{x_1,\dots,x_k}\in F_k$ for $k\ge 3$, with $-\infty=x_0<x_1<\dots<x_k<x_{k+1}=\infty$. Consider $i\in\{1,\dots,k-1\}$ such that $x_{i+1}-x_i< x_{j+1}-x_j$ for all $j\neq i$ (we may assume that the inequality is strict by perturbing the $x_i$ and taking a limit if necessary). We consider the following cases.
\begin{itemize}
    \item If $x_{i+1}\le 0$, then set
    \[x'_j=\begin{cases}x_j+2(x_{i+1}-x_i)&1\le j<i,\\
    x_{j+2}&i\le j\le k-2.
    \end{cases}\]
    \item If $x_{i}\ge 0$, then set
    \[x'_j=\begin{cases}
    x_j&1\le j<i,\\
    x_{j+2}-2(x_{i+1}-x_i)&i\le j\le k-2.
    \end{cases}\]
    \item If $x_i<0<x_{i+1}$, then set
    \[x'_j=\begin{cases}
    x_j-x_i&1\le j<i,\\
    x_{j+2}-x_{i+1}&i\le j\le k-2.
    \end{cases}\]
\end{itemize}
Then it suffices to prove that
\[\phi=\phi_{x'_1,\dots,x'_{k-2}}\circ\phi_{x_i,x_{i+1}}.\]
To do this, we verify \cref{eq:def:phi} in each case. We will only treat the case $x_i\ge 0$, the others two being analogous. We have that 
\begin{equation}
\label{eq:derivatives}
\phi'_{x'_1,\dots,x'_{k-2}}(\phi_{x_i,x_{i+1}}(x))\times \phi'_{x_i,x_{i+1}}(x)
\end{equation}
changes sign at $x_i$ and $x_{i+1}$ due to the second factor. Moreover, $\phi_{x_i,x_{i+1}}$ takes the values in $I=\bbR\setminus[2x_i-x_{i+1},x_i]$ exactly once and
\[\phi_{x_i,x_{i+1}}(x_j)=\begin{cases}x'_{j}&1\le j< i,\\
x'_{j-2}& i+2\le j\le k.\end{cases}\]
But our choice of $i$ implies $I\supset\{x'_1,\dots,x'_{k-2}\}$,
so the first factor in \cref{eq:derivatives} changes sign precisely at $x_1,\dots,x_{i-1},x_{i+2},\dots,x_k$, concluding the proof.
\end{proof}
With \cref{prop:reduction} it is immediate to deduce \cref{lemma1}.
\begin{proof}[Proof of \cref{lemma1}]
Observe that $G=\{\mathrm{id},-\mathrm{id}\}\circ(F_0\cup F_1\cup F_2)$. Thus,
\begin{equation}
\label{eq:G}
\langle G\rangle\supset\{\mathrm{id},-\mathrm{id}\}\circ\langle F_2\rangle\circ(F_0\cup F_1)\supset\{\mathrm{id},-\mathrm{id}\}\circ\bigcup_{k=0}^\infty F_k\supset\langle G\rangle,
\end{equation}
where the first and third inclusions follow by definition, while the second one is \cref{prop:reduction}. Thus, $\langle G\rangle=\{\mathrm{id},-\mathrm{id}\}\circ\bigcup_{k=0}^\infty F_k$. It therefore remains to show that $\langle G\rangle$ is dense in $\Phi$, in order to conclude the proof.

To this extent, note that any $\phi \in \Phi$ coincides with its $1-$Lipschitz envelope, i.e., 
\[\phi(x) = \inf_{y\in\mathbb R} \phi(y) + |x-y|, \quad \forall x \in \mathbb{R}.\]
By continuity, 
\[\phi(x) = \inf_{y\in \mathbb{Q}} \phi(y) + |x-y|, \quad \forall x \in \mathbb{R}.\]
Taking a sequence of  finite sets $(Q_n)_n \uparrow \mathbb{Q}$  with $Q_0=\{0\}$, we can approximate $\phi$ with $\phi_n\in -\mathrm{id}\circ F_{2k_n-1}$ for some $k_n\in\{1,\dots,|Q_n|\}$ given by
\[\phi_n(x) := \inf_{y \in Q_n} \phi(y) + |x-y|, \quad \forall x \in \mathbb{R}.\] The limit $\phi_n \to \phi$ is in uniform convergence  on compact sets thanks to equi-continuity, so the proof is complete.
\end{proof}
We are ready to assemble the proof of \cref{th:contraction}.
\begin{proof}[Proof of \cref{th:contraction}]
By \cref{thm1}, any non-bilinear Dirichlet form $\cE$ satisfies \cref{eq:HK,eq:minmax} and is l.s.c. Since symmetry is a hypothesis of \cref{th:contraction}, together with \cref{prop:two-sided,prop:one-sided,prop:one-cusp} it yields that for any $\phi\in G$ (recall \cref{lemma1}) and $f\in\mathrm{L}^2(X,m)$ it holds that $\cE(\phi\circ f)\le \cE(f)$. Indeed, $F_0$ is trivial, \cref{prop:one-cusp} deals with $F_1$, \cref{prop:one-sided,prop:two-sided} give $F_2$ and then symmetry allows us to take opposites. Therefore, the normal contraction property \cref{eq:normalcontraction} also holds for all $\phi\in\langle G\rangle$. 

Fix $f\in\mathrm{L}^2(X,m)$ and an arbitrary normal contraction $\phi\in\Phi$. By \cref{lemma1}, there exists a sequence $\phi_n\in\langle G\rangle$ such that $\phi_n(x) \to \phi(x)$ for all $x \in \mathbb{R}$, as $n\to\infty$, and 
\[\E(\phi_n\circ f) \leq \E(f)\] for all $n$. We have that $\phi_n(f) \to \phi(f)$ pointwise in $X,$ but 
\[|\phi_n \circ f|^2 \leq |f|^2 \in \mathrm{L}^1(X,m),\]
as all functions $\phi_n$ are normal contractions. Then, by Lebesgue's dominated convergence theorem
\[\phi_n\circ f \to \phi \circ f\]
in $\mathrm{L}^2(X,m).$ Thus, we obtain the desired inequality via the l.s.c.\ of $\E$. 
\end{proof}

\subsection{Locality}
\label{sec:local}
Let us conclude this section with a concept of locality allowing a much more direct proof of \cref{th:contraction} under this hypothesis. We say that a non-bilinear Dirichlet form $\E$ is \emph{local} if for all $c\in\bbR$ and $u,v \in \mathrm{L}^2(X,m)$ such that $u(x)(v(x)-c)=0$ for all $x\in X$, we have 
\[\diri(u+v)=\diri(u)+\diri(v).\]
\begin{proof}[Proof of \cref{th:contraction} in the local case]
Fix a symmetric local non-bilinear Dirichlet form $\cE$. As in the proof of \cref{th:contraction} it suffices to establish the normal contraction property \cref{eq:normalcontraction} for all $\phi\in\bigcup_{k=1}^\infty F_k$ (this part of the proof does not rely on \cref{thm1,prop:one-cusp,prop:one-sided,prop:two-sided,prop:reduction}). Fix $\phi=\phi_{x_1,\dots,x_k}$ for some $x_1<\dots<x_k$. Observe that 
\[\phi(x)=((x-x_1)\wedge 0)+\sum_{i=1}^k (-1)^i((0\vee (x-x_i))\wedge (x_{i+1}-x_i)).\]
Since all summands satisfy the locality condition, we get 
\begin{align*}
\cE(\phi\circ u)&{}=\cE((u-x_1)\wedge 0))+\sum_{i=1}^k \cE\left((-1)^i((0\vee (u-x_i))\wedge (x_{i+1}-x_i))\right)\\
&{}=\cE((u-x_1)\wedge 0))+\sum_{i=1}^k \cE((0\vee (u-x_i))\wedge (x_{i+1}-x_i))=\cE(u),\end{align*}
using symmetry and locality for the second and third equalities.
\end{proof}

\section{Future directions}
\label{sec:future}
Two challenges which are still open are the following. Firstly, we are not aware of any attempt to obtain a structural decomposition analogous to the one of \cite{fukushima2011dirichlet} 
in the non-bilinear setting. Secondly, the theory of \cite{lott2009ricci,sturm2006geometryI,sturm2006geometryII} covers even the case where Cheeger's energy of the metric measure space is a non-bilinear form, while an analogue of \cite{ambrosio2015bakry} for the non-bilinear case is missing. It is our opinion that the subject of metric measure spaces would profit from a study in this direction. 

These two problems are strong motivations behind our paper, as we foresee that the normal contraction property would be crucial in developing such theories. One difficulty we anticipate is the generalisation of the computations in \cite{bakry1985diffusions}, which looks complicated even in the case of Finsler manifolds. Finally, establishing the normal contraction property adds one structural argument in favour of the choice made by Cipriani and Grillo of the generalisation of bilinear Dirichlet forms to the non-bilinear setting.

\let\d\oldd
\let\k\oldk
\let\l\oldl
\let\L\oldL
\let\o\oldo
\let\O\oldO
\let\r\oldr
\let\S\oldS
\let\t\oldt
\let\u\oldu

\section*{Declarations}
The first author has been funded by the European Union’s Horizon 2020 research and innovation program under the Marie Skłodowska-Curie grant agreement No 754362. Partial support has been obtained from the EFI ANR-17-CE40-0030 Project of the French National Research Agency. The second author was supported by ERC Starting Grant 680275 ``MALIG.'' \\
We are grateful to Giuseppe Savar\'e for numerous stimulating and helpful comments  and to the anonymous referee for useful suggestions. We thank the organisers of the CEREMADE Young Researcher Winter School 2022, where part of the work was done.
\\
Data sharing not applicable to this article as no datasets were generated or analysed during the current study.
\\
The authors have no competing interests to declare that are relevant to the content of this article.

\bibliography{ref}
\bibliographystyle{siam}
\end{document}